\documentclass[a4paper,11pt]{article}
\usepackage[a4paper]{geometry}
\usepackage[OT2,T1]{fontenc}
\usepackage{lipsum}
\usepackage[english]{babel}
\usepackage{amssymb,amsmath,amsthm,amscd,amsfonts}
\usepackage{mathrsfs}
\usepackage{tikz-qtree}
\usetikzlibrary{shapes,arrows.meta}
\usepackage{subfig}
\usepackage{tabularx}
\usepackage{calc} 
\usepackage{graphicx}
\usepackage[nottoc]{tocbibind}
\usepackage{makeidx}
\usepackage{authblk}
\makeindex
\usepackage{longtable}
\usepackage{enumerate}
\usepackage{array}
\usepackage{mathrsfs}
\usepackage{version}
\usepackage{lscape}
\usepackage{epigraph}
\usepackage[all]{xy}
\usepackage{datetime}
\usepackage{multirow}
\usepackage{pgf,tikz}
\usepackage{hyperref}   
\usepackage{etoolbox}
\usepackage{lmodern}
\usepackage[font=scriptsize]{caption}
\apptocmd{\thebibliography}{\raggedright}{}{}

\newtheoremstyle{one}
{11pt}
{11pt}
{\it}
{}
{\bf}
{.}
{1mm}
{}

\newtheoremstyle{two}
{11pt}
{11pt}
{}
{}
{\bf}
{.}
{1mm}
{}

\theoremstyle{one}
\newtheorem{theorem}{Theorem}[section]
\newtheorem{lemma}[theorem]{Lemma}
\newtheorem*{lemme}{Lemme}
\newtheorem*{theoreme}{Th\'eor\`eme}
\newtheorem{proposition}[theorem]{Proposition}

\theoremstyle{two}

\newtheorem{definition}[theorem]{Definition}
\newtheorem{example}[theorem]{Example}

\newtheorem{notation}[theorem]{Notation}
\newtheorem{remark}[theorem]{Remark}
\newtheorem*{prooffamilymultisections}{\textbf{Proof of Proposition \ref{multisection}}}
\newtheorem*{proofofmainthm}{\textbf{Proof of Theorem \ref{thm1}}}

\def\Q{\mathbb{Q}}

\def\P{\mathbb{P}}

\title{Density of rational points on a family of del Pezzo surfaces of degree one}

\author{Julie Desjardins and Rosa Winter \\

(With an appendix by Jean-Louis Colliot-Th\'{e}l\`{e}ne)}
\date{}

\begin{document}
\maketitle

\begin{center}\textbf{Keywords:} Rational points, del Pezzo surfaces, elliptic fibrations \end{center}

\abstract{Let $k$ be an infinite field of characteristic 0, and $X$ a del Pezzo surface of degree $d$ with at least one $k$-rational point. Various methods from algebraic geometry and arithmetic statistics have shown the Zariski density of the set $X(k)$ of $k$-rational points in $X$ for $d\geq2$ (under an extra condition for $d=2$), but fail to work in generality when the degree of $X$ is 1, leaving a large class of del Pezzo surfaces for which the question of density of rational points is still open. In this paper, we prove the Zariski density of $X(k)$ when $X$ has degree 1 and is represented in the weighted projective space $\mathbb{P}(2,3,1,1)$ with coordinates $x,y,z,w$ by an equation of the form $y^2=x^3+az^6+bz^3w^3+cw^6$ for $a,b,c\in k$ with $a,c$ non-zero, under the condition that the elliptic surface obtained by blowing up the base point of the anticanonical linear system $|{-K}_X|$ contains a smooth fiber above a point in $\mathbb{P}^1\setminus\{(1:0),(0:1)\}$ with positive rank over $k$. When $k$ is of finite type over $\mathbb{Q}$, this condition is sufficient \textsl{and} necessary.  
}
\setcounter{tocdepth}{2}

\section{Introduction}

A del Pezzo surface over a field $k$ is a smooth, projective, geometrically integral surface over $k$ with ample anticanonical divisor. The degree of a del Pezzo surface is the self-intersection number of the canonical divisor, and this is an integer between 1 and 9.
Over an algebraically closed field, a del Pezzo surface of degree $d$ is isomorphic to either $\mathbb{P}^1\times\mathbb{P}^1$ (for $d=8$), or to $\mathbb{P}^2$ blown up in $9-d$ points in general position. 
Over a non-algebraically closed field, this is not true in general. A variety $X$ over a field $k$ is \textsl{$k$-unirational} if there is a dominant rational map $\mathbb{P}_k^n\dashrightarrow X$ for some $n$. Del Pezzo surfaces of degree at least 2 over a field $k$ with a $k$-rational point are known to be $k$-unirational under the extra condition for degree 2 that the $k$-rational point lies outside the ramification curve of the anticanonical map, and is not contained in the intersection of 4 exceptional curves. This is proved consecutively in \cite{Se43} and \cite{Se51} for degree 3 and $k=\mathbb{Q}$, in \cite[Theorems 29.4, 30.1]{Ma86} for $d\geq5$, as well as for $d=3,4$ for large enough cardinality of $k$, in \cite{Ko02} for the complete case $d=3$, in \cite[Proposition 5.19]{Pie} for the complete case $d=4$, and in \cite{STVA} for $d=2$. A del Pezzo surface is minimal if and only if there exists no birational map over its groundfield to a del Pezzo surface of higher degree. Therefore, if a del Pezzo surface $X$ of degree 1 over a field $k$ is not minimal, it is birationally equivalent to a del Pezzo surface $X'$ of higher degree, and $X$ is unirational if and only if $X'$ is. A minimal del Pezzo surface of degree 1 has Picard rank 1 or 2. \\
For a long time, nothing about unirationality for minimal del Pezzo surfaces of degree 1 was known, even though they always contain a rational point. In 2017, Koll\'{a}r and Mella proved that minimal del Pezzo surfaces of degree 1 over a field $k$ with char $k\neq2$ that have Picard rank 2 are $k$-unirational \cite{KM17}. Outside this case, the question of $k$-unirationality for minimal del Pezzo surfaces of degree~1 is wide open: we do not have any example of a minimal del Pezzo surface of degree 1 with Picard rank 1 that is known to be $k$-unirational, nor of one that is known not to be $k$-unirational. \\
If $k$ is infinite, then $k$-unirationality implies density of the set of $k$-rational points. While unirationality for del Pezzo surfaces of degree 1 is still out of reach, there are several partial results on density of their set of rational points; see Remark \ref{previouswork}. Moreover, if $k$ is a number field, density of the set of $k$-rational points for del Pezzo surfaces of degree 1 is implied by a conjecture of Colliot-Th\'{e}l\`{e}ne and Sansuc, stating that for a geometrically rational variety over a number field, its set of rational points is dense in the Brauer--Manin set for the adelic topology~\cite[Question $(\textbf{j}_1)$]{CTS}.\\
In this paper we give sufficient conditions for del Pezzo surfaces of degree~1 in a certain family over an infinite field of characteristic 0 to have a dense set of rational points. The conditions are necessary if the field is of finite type over $\mathbb{Q}$. 

\subsection{Main result}
A del Pezzo surface of degree 1 over a field $k$ can be described by a smooth sextic in the weighted projective space $\mathbb{P}(2,3,1,1)$ with coordinates $(x:y:z:w)$. For char $k\neq2,3$, this sextic can be written as \begin{equation}\label{eqDP1}y^2=x^3+x\cdot f(z,w)+g(z,w),
\end{equation} where $f,g\in k[z,w]$ are homogeneous of degree 4 and 6, respectively. For a del Pezzo surface $X$ of degree 1, the anticanonical linear system $|{-K}_X|$ has a unique base point given by $\mathcal{O}=(1:1:0:0)$. Blowing up this basepoint gives a surface $\mathcal{E}$ with elliptic fibration $\mathcal{E}\longrightarrow\mathbb{P}^1$, which, when restricted to $S$, is given  by the projection to $(z:w)$. The fibration admits a section $\tilde{\mathcal{O}}$ given by the exceptional curve above $\mathcal{O}$. 

In this paper we prove the following theorem.

\begin{theorem}\label{thm1}
Let $k$ be a field of characteristic 0, and $a,b,c\in k$ with $a,c$ non-zero. Let $S$ be the del Pezzo surface given by \begin{equation}\label{eqdp1thm}y^2=x^3+az^6+bz^3w^3+cw^6
\end{equation} in the weighted projective space $\mathbb{P}(2,3,1,1)$ with coordinates $(x,y,z,w)$. Let $\mathcal{E}$ be the elliptic surface obtained by blowing up the base point of the linear system $|{-K}_S|$. If $S$ contains a rational point with non-zero $z,w$-coordinates, such that the corresponding point on~$\mathcal{E}$ is non-torsion on its fiber, then $S(k)$ is dense in $S$ with respect to the Zariski topology. If $k$ is of finite type over $\mathbb{Q}$, the converse holds as well.
\end{theorem}

\begin{remark}
Theorem \ref{thm1} is the first result that gives sufficient \textsl{and} necessary conditions for the $k$-rational points on the family given by (\ref{eqdp1thm}) to be dense, even when $b=0$, where $k$ is any field of finite type over $\mathbb{Q}$; see also Remark~\ref{previouswork}. We require $k$ to be of finite type over $\mathbb{Q}$ in order to bound the torsion in a family of elliptic curves over $k$; see also Theorem \ref{CT}.  
\end{remark}

\begin{remark}\label{previouswork}
Several partial results on density of rational points on a del Pezzo surface of degree~1 are known. In \cite{VA11}, V\'{a}rilly-Alvarado proves Zariski density of the set of $\mathbb{Q}$-rational points on all surfaces of the form (\ref{eqDP1}) with $f=0$ and $g=az^6+bw^6$ for non-zero $a,b\in\mathbb{Z}$, such that either $3a/b$ is not a square, or gcd($a,b)=1$ and $9\nmid ab$, under the condition that the Tate--Shafarevich group of elliptic curves with $j$-invariant 0 is finite. Ulas and Togb\'e prove Zariski density of the set of $\mathbb{Q}$-rational points of surfaces of the from (\ref{eqDP1}) in the following cases. (i) either $g=0$ and deg$(f(z,1))\leq3$, or $g=0$ and deg$(f(z,1))=4$ with $f$ not even, or $f=0$ and $g(z,1)$ is monic of degree 6 and not even \cite[Theorems 2.1 (1), 2.2, and 3.1]{U07}. (ii) $g=0$ and deg$(f(z,1))=4$, or $f=0$ and $g(z,1)$ is even and monic of degree 6, both cases under the condition that there is a fiber of $\mathcal{E}$ with infinitely many rational points \cite[Theorems 2.1 (2) and 3.2]{U07}. (iii)
The surface can be defined by $y^2=x^3-h(z,w)$, with $h(z,1)=z^5+az^3+bz^2+cz+d\in\mathbb{Z}[z]$, and the set of rational points on the elliptic curve $Y^2=X^3+135(2a-15)X-1350(5a+2b-26)$ is infinite \cite[Theorem 2.1]{U08}. (iv) $f(z,1)$ and $g(z,1)$ are both even of degree 4 and there is a fiber of $\mathcal{E}$ with infinitely many rational points \cite[Theorem 2.1]{UT10}. Jabara generalized the results from \cite{U07} mentioned above in \cite[Theorems C and D]{J12}. Though the proofs of these two theorems are incomplete (see \cite[Remark 2.7]{SL14}), they hold for sufficiently general cases. Finally, in \cite{SL14}, Salgado and van Luijk generalize some of the previous results, proving Zariski density of the set of $k$-rational points of surfaces of the form (\ref{eqDP1}) for any infinite field $k$ with char $k\neq2,3$, assuming that there exists a point $Q$ on a smooth fiber of $\mathcal{E}$ satisfying several conditions, among which that a multisection that they construct from $Q$ has infinitely many $k$-rational points. 
\end{remark}

\begin{remark}
For an elliptic surface, the Zariski density of the set of rational points is equivalent to having infinitely many fibers with non-zero rank. Given the difficulty in calculating the rank in generality, a reasonable substitute when $k=\Q$ is the root number $W(E)$, defined as the parity of the analytic rank. A consequence of the Birch and Swinnerton-Dyer conjecture (known as the Parity conjecture) relates the root number to the parity of the geometric rank by $W(E)=(-1)^{r(E)}$, reducing the question of Zariski density on elliptic surfaces to finding an infinite set of fibers with root number $-1$, which would imply odd and thus non-zero rank. Work in this direction includes papers from Manduchi \cite{Manduchi}, Helfgott \cite{Helfgott} and V\'{a}rilly-Alvarado (mentioned in the previous remark). The latest result in the literature \cite{Desjardins1}, by the first author, proves that for non-isotrivial elliptic surfaces, the set of fibers with negative root number is infinite, assuming the Chowla conjecture on the product of polynomials corresponding to places of multiplicative reduction, and the Squarefree conjecture on certain\footnote{To be precise, the polynomials on which we need the Squarefree conjecture are the non-insipid ones, according to the vocabulary introduced in \cite{Desjardins5}.} polynomials associated to places of bad reduction. This proves Zariski density conditionally on these two conjectures and on the Parity conjecture for all non-isotrivial elliptic surfaces. The Chowla and Squarefree conjectures are known to hold for polynomials of small degree, rendering Desjardins's work on finding the root number unconditional for a lot of rational elliptic surfaces.

However, on isotrivial elliptic surfaces it can happen that every fiber has a positive root number, in which case one cannot use this quantity to predict whether the geometric rank is non-zero on infinitely many fibers. When the $j$-invariant of the generic fiber is non-zero and the elliptic surface is rational, the first author proves in \cite[Theorem 1.2]{Desjardins2} $\Q$-unirationality when $j\not=1728$, and Zariski density otherwise, both unconditionally. 

A rational elliptic surface with $j$-invariant zero and only irreducible singular fibers is birationally equivalent to a del Pezzo surface of degree 1, by blowing down the zero section. Combining \cite[Theorem 2.1]{VA11} and \cite[Proposition 2.2]{DN}, we know that the elliptic surfaces with $j=0$ and root number $+1$ on every fiber are of the form $y^2=x^3+c(f(t)^2+3g(t)^2)$, for some $c\in\mathbb{Q}$, and $f,g\in\Q(t)$ non-zero and with no common factors. In particular, all examples in Section \ref{examples} are in this family. Sufficient conditions for such surfaces to have root number $+1$ on all fibers are given in \cite{Desjardins2} in the case $f(t)=At^3$, $g(t)=B$ for some $A,B\in\mathbb{Z}$ non-zero. Finally, in the paper \cite{DN} by Naskr\c ecki and the first author, the generic rank of all surfaces given by an equation $y^2=x^3+at^6+c$ for $a,c\in\mathbb{Z}$ non-zero is computed, proving $\Q$-unirationality for those with non-zero generic rank. However, it is found that most of these surfaces have rank zero.
\end{remark}

\subsection{Set-up and idea of the proof}We set up some terminology that we will use throughout the paper. 
Let $C$ be a smooth, projective, geometrically integral curve over a field $k$. An \textsl{elliptic surface} with base $C$ is a smooth, projective, geometrically integral surface $\mathcal{E}$ endowed with an \textsl{elliptic fibration} $\pi:\mathcal{E}\longrightarrow C$, that is, a surjective morphism such that for almost all points $v\in C$, the fiber $\mathcal{E}_v=\pi^{-1}(v)$ above $v$ is a smooth genus~1 curve, and moreover, the morphism $\pi$ admits a \textsl{section}: a morphism $s\colon C\longrightarrow \mathcal{E}$ such that $\pi\circ s=\mbox{id}_{C}$. The existence of a section implies that all smooth genus 1 fibers of $\pi$ are elliptic curves. If $\mathcal{E}$ is an elliptic surface with base $C$, then its \textsl{generic fiber} is an elliptic curve $E$ over the function field $k(C)$ of $C$. The set of sections of $\pi:\mathcal{E}\longrightarrow C$ form a group which is in natural correspondence with the group of $k(C)$-rational points of $E$ \cite[Proposition 5.4]{MW19}, also called the \textsl{Mordell--Weil} group of $\mathcal{E}$. A multisection of degree $d$ or $d$-section of $\mathcal{E}$ is an irreducible curve $D$ contained in $\mathcal{E}$ such that the projection $\pi_D\colon D\longrightarrow C$ is non-constant and of degree $d$. Note that, when $d=1$, this is simply a section. We often switch between viewing a (multi-)section as a map between schemes and a curve on $\mathcal{E}$. 

\begin{remark}\label{sectionimpliesdensity}
Let $k, S$, and $\mathcal{E}$ be as in Theorem \ref{thm1}. If $\mathcal{E}$ admits a $k$-rational section $s$, then, since the Mordell-Weil group of $\mathcal{E}$ is torsion-free \cite[Theorem 7.4]{MW19}, there are infinitely many distinct multiples of $s$ contained in $\mathcal{E}$, all of which have a dense set of rational points. This implies that $\overline{\mathcal{E}(k)}$ contains an infinite union of distinct one-dimensional irreducible closed subsets, hence $\overline{\mathcal{E}(k)}=\mathcal{E}$, and thus the Zariski density of $\mathcal{E}(k)$ follows, implying the density of $S(k)$ in $S$. Theorem \ref{thm1} is therefore especially interesting in the cases where the Mordell-Weil group of $\mathcal{E}$ has rank zero over~$k$. 
\end{remark}

We prove Theorem \ref{thm1} by constructing, for $k$, $S$, and $\mathcal{E}$ as in the theorem, a family of multisections of $\mathcal{E}$, and using these multisections to show that the set $\mathcal{E}(k)$ is dense in $\mathcal{E}$.

\subsection{Contents of the paper}

The paper is organized as follows. In Section \ref{construction} we construct a family of multisections on $\mathcal{E}$, and show that there is a member of this family that contains a section over $k$, or there is a member of this family that is geometrical integral of genus 0, or there are infinitely many members of this family that are geometrically integral of genus 1. In Section \ref{familyofmultisectionsnontorsion} we show that, in the latter case, this family gives rise to an elliptic fibration over a smooth fiber of $\mathcal{E}$, and that this elliptic fibration has a non-torsion section (see also Figure \ref{fiberfibration}). In Section \ref{proofmainthm} we use all this to prove Theorem~\ref{thm1}. Finally, in Section \ref{examples} we give examples of specific surfaces for which Theorem \ref{thm1} shows that the set of rational points is dense. 

\vspace{11pt}

\textbf{Acknowledgements.} We thank Ronald van Luijk for giving us the idea to look for certain multisections, and for many useful ideas and discussions afterwards. We thank Jean-Louis Colliot-Th\'{e}l\`{e}ne for showing us how to extend the necessary condition over a number field to any field of finite type over $\mathbb{Q}$.
We are grateful to Martin Bright, J\'{a}nos Koll\'{a}r, Bartosz Naskr\c ecki, Marta Pieropan, and Anthony V\'arilly-Alvarado for helpful discussions and remarks. We thank the anonymous referee for useful comments.

\section{Constructing a family of multisections}\label{construction}

Let $k$ be an infinite field with char $k\neq2,3$, take $a,b,c\in k$ with $a,c$ non-zero, let $S$ be the del Pezzo surface of degree 1 given by (\ref{eqdp1thm}) with canonical divisor $K_S$, and assume that $S$ contains a rational point as in Theorem \ref{thm1}. Let $\mathcal{E}$ be the elliptic surface obtained by blowing up the base point of the linear system $|{-K}_S|$. In this section we construct a family of multisections on $\mathcal{E}$; see Proposition~\ref{multisection}. We introduce some notation.

\begin{notation}\label{setup}Let $\pi\colon\mathcal{E}\longrightarrow S$ be the blow-up of $S$ in $\mathcal{O}=(1:1:0:0)$ with exceptional divisor~$\tilde{\mathcal{O}}$. Since~$\pi$ gives an isomorphism between $\mathcal{E}\setminus\tilde{\mathcal{O}}$ and $S\setminus\{\mathcal{O}\}$, we denote a point $R\in\mathcal{E}\setminus\tilde{\mathcal{O}}$ by the coordinates of $\pi(R)$ in $\mathbb{P}_k(2,3,1,1)$. Let $\nu\colon\mathcal{E}\longrightarrow\mathbb{P}^1$ be the elliptic fibration on $\mathcal{E}$, which is given on $S$ by the projection onto $(z:w)$. For $R=(x_R:y_R:z_R:w_R)\in S\setminus\{\mathcal{O}\}$, we denote by $R_{\mathcal{E}}$ the inverse image $\pi^{-1}(R)$ on~$\mathcal{E}$, which is a point on the fiber $\nu^{-1}((z_R:w_R))$.\end{notation}

\begin{definition}\label{defncurve}
For any point $R=(x_R:y_R:z_R:1)$ in $\mathcal{E}$ with $y_R,z_R\neq0$, we define the curve $C_R\subset \mathcal{E}$ as the strict transform of the intersection of $S$ with the surface given by  
\begin{equation}\label{C} 
3x_R^2z_R^2xz - 2y_Rz_R^3y - (x_R^3-2az_R^6 - bz_R^3)z^3 +(2cz_R^3+bz_R^6)w^3=0.
\end{equation}
\end{definition}

\begin{remark}\label{kollar}
The curve $C_R$ in (\ref{C}) was found by finding generators for the subsytem of $|{-3K}_S|$ of curves that contain $R$ as a double point. A general such curve has arithmetic genus 3, and by forcing two extra singularities we hoped to find curves of genus at most 1. When computing the generators in \texttt{magma}, the curve $C_R$ was one of them. Later, the following was pointed out to us by J\'{a}nos Koll\'{a}r. By setting $z=1,\;x'=\tfrac{x}{z^3},\;y'=\frac{y}{z^3},\;w'=\tfrac{w}{z}$, we obtain the affine model of $S$ in $\mathbb{A}^3$ given by $y'^2=x'^3+a+bw'^3+cw'^6$. Setting $\omega=w'^3$, we obtain a map $f$ from this affine model to the cubic surface $T$ given by $y'^2=x'^3+a+b\omega+c\omega^2$ in $\mathbb{A}^3$ with coordinates $x',y',\omega$. For a point $R\in S$ with non-zero $y,z,w$-coordinates, the tangent plane to $T$ at $Q=f(R)$ intersects $T$ in a cubic curve $C_Q$, and the curve $C_R$ as in Definition \ref{defncurve} is the strict transform on $\mathcal{E}$ of the preimage of $C_Q$ under $f$.
\end{remark}

\begin{remark}
For $R=(x_R:y_R:z_R:1)$ in $\mathcal{E}$ with $y_R,z_R\neq0$,  the curve $\pi(C_R)$ does not contain the point $\mathcal{O}$, so we identify the curve $C_R$ with $\pi(C_R)\subset\mathbb{P}(2,3,1,1)$; see Notation \ref{setup}. 
\end{remark} 

The main result in this section is the following. We prove this at the end of the section. 

\begin{proposition}\label{multisection}Let $\mathcal{F}$ be a smooth fiber on $\mathcal{E}$ above a point $(z_0:1)\in\mathbb{P}^1$ with $z_0\in k$ non-zero, such that $\mathcal{F}$ has non-zero rank over $k$. Then there is a point $R\in\mathcal{F}(k)$ such that $C_R$ contains a section defined over $k$, or there is a point $R\in\mathcal{F}(k)$ such that $C_R$ is geometrically integral of genus~0, or there is a non-empty open subset $\mathcal{F}_0$ of $\mathcal{F}$ such that for every point $R\in\mathcal{F}_0(k)$, the curve $C_R$ is geometrically integral of genus 1.
\end{proposition}

\begin{remark}\label{affinespace}
Let $R=(x_R:y_R:z_R:1)$ be a point in $\mathcal{E}$, with $y_R,z_R\neq0$, and let $C_R$ be the corresponding curve as in Definition \ref{defncurve}. Let $\mathbb{A}^3$ be the affine open subset of $\mathbb{P}(2,3,1,1)$ given by $w\neq0$, with coordinates $X=\frac{x}{w^2}$, $Y=\frac{y}{w^3}$, and $T=\frac{z}{w}$. We describe the intersection $C_R\cap\mathbb{A}^3$. Write 
\begin{align}\label{Caffien}
F&=Y^2-X^3-aT^6-bT^3-c,\\
G&=3x_R^2z_R^2XT- 2y_Rz_R^3Y - (x_R^3-2az_R^6 - bz_R^3)T^3 +2cz_R^3+bz_R^6.\nonumber
\end{align}
We have $C_R\cap\mathbb{A}^3=Z(F)\cap Z(G)$, where $Z(F)$ and $Z(G)$ are the zero loci of $F$ and $G$, respectively. Since $y_R,z_R\neq0$, the projection $p\colon \mathbb{A}^3\longrightarrow\mathbb{A}^2$ to the $X,T$-coordinates has a section given by $$r\colon(X,T)\longmapsto\left(X,\frac{3x_R^2z_R^2XT- (x_R^3-2az_R^6 - bz_R^3)T^3 +2cz_R^3+bz_R^6}{2y_Rz_R^3},T\right).$$ Note that $p$ induces an isomorphism $Z(G)\longrightarrow \mathbb{A}^2$ with inverse $r$. It follows that $C_R\cap\mathbb{A}^3$ is isomorphic to $p(Z(F)\cap Z(G))$, and the latter is defined by $H_R=0$, where
\begin{align}\label{H}
H_R&=4y_R^2z_R^6X^3  -9x_R^4z_R^4X^2T^2 + ( 6x_R^5z_R^2-12ax_R^2z_R^8 - 6bx_R^2z_R^5)XT^4 \nonumber \\
&- (12cx_R^2z_R^5 +6bx_R^2z_R^8)XT + (4acz_R^6+8ax_R^3z_R^6- b^2z_R^6 + 2bx_R^3z_R^3 -x_R^6)T^6 \nonumber\\
&- 2(4acz_R^9 -2cx_R^3z_R^3-3bx_R^3z_R^6 - b^2z_R^9)T^3 + 4acz_R^{12}+4cx_R^3z_R^6- b^2z_R^{12}
.
\end{align}
\end{remark}

We denote by $K_{\mathcal{E}}$ the canonical divisor of $\mathcal{E}$. Let $\overline{k}$ be an algebraic closure of $k$, and write $\overline{C}_R$ for the base change $C_R\times_k\overline{k}$. Recall that $\nu\colon\mathcal{E}\longrightarrow \mathbb{P}^1$ is the elliptic fibration on $\mathcal{E}$ (Notation~\ref{setup}).

\begin{lemma}\label{basicfactscr}
Let $R=(x_R:y_R:z_R:1)$ be a point in $\mathcal{E}$ with $y_R,z_R\neq0$, and let $C_R$ be the curve in Definition \ref{defncurve}. The following hold. 
\begin{itemize}
\item[](i) The curve $C_R$ does not contain a fiber of $\mathcal{E}$. 
\item[](ii) The curve $C_R$ is contained in the linear system $|{-3K}_{\mathcal{E}}+3\tilde{\mathcal{O}}|$, and intersects every fiber of $\nu$ in three points counted with multiplicity.
\end{itemize}
\end{lemma}
\begin{proof}
(i). From (\ref{C}) it is clear that $C_R$ does not contain the fiber $w=0$. Moreover, since the coefficient of $X^3$ of $H_R$ in (\ref{H}) as a polynomial in $k[T]$ is constant and non-zero, $C_R$ does not contain any fiber with $w\neq0$, either. \\
(ii). The linear system $|{-3K}_S|$ induces the 3-uple embedding of $S$ into $\mathbb{P}^6$ \cite[page 1200]{CO99}. Under this embedding, the curve $\pi(C_R)$ is given by the intersection of $S$ with a hyperplane, hence we have $\pi(C_R)\sim {-3K}_S$. 
Since $y_R,z_R\neq0$, the image $\pi(C_R)$ does not contain the point $\mathcal{O}$, so this implies 
$$C_R=\pi^*(\pi(C_R))\in|\pi^*({-3K}_S)|=|{-3K}_{\mathcal{E}}+3\tilde{\mathcal{O}}|.$$ Since a fiber $F$ of $\nu$ is linearly equivalent to $-K_{\mathcal{E}}$, which has self-intersection $(\pi^*({K}_S)-\tilde{\mathcal{O}})^2=0$, and $\tilde{\mathcal{O}}$ is a section of $\nu$, we have $F\cdot C_R=F\cdot({-3K}_{\mathcal{E}}+3\tilde{\mathcal{O}})=0+3=3.$ Since $F$ is irreducible, it follows that, since $F$ is not contained in $C_R$, the number of intersection points of $F$ and $C_R$ is 3, counted with multiplicity.
\end{proof}

In Proposition \ref{singularpointsC} we describe for which choices of $R\in\mathcal{E}$ the curve $C_R$ has genus at most~1. In the proof we use several known results on the exceptional curves on $S$, which we state in the following remark.

\begin{remark}\label{exccurvessections}The surface $S$ contains 240 exceptional curves, which are defined over the separable closure $k^{\mbox{\tiny{sep}}}$ of $k$ in $\overline{k}$; this follows from \cite[Propositions 5 and 7]{Co88}, see for example \cite[Theorem 2.1.1]{VA09}. Therefore, from \cite[Theorem 1.2]{VA08} it follows that the exceptional curves on $S$ are exactly the curves given by $$x=p(z,w),\;\;\;y=q(z,w),$$ where $p,q\in k[z,w]$ are homogeneous of degrees 2 and 3. Note that this implies that an exceptional curve never contains $\mathcal{O}=(1:1:0:0)$. Therefore, for an exceptional curve $C$ on $S$, its strict transform $\pi^*(C)$ on $\mathcal{E}$ satisfies $$\pi^*(C)^2=-1,\;\;\;\;\;\pi^*(C)\cdot-K_{\mathcal{E}}=\pi^*(C)\cdot(\pi^*({K}_S)+\tilde{\mathcal{O}})=1+0=1,$$ so $\pi^*(C)$ is an exceptional curve on $\mathcal{E}$ as well. Moreover, since a fiber of $\nu$ is linearly equivalent to $-K_{\mathcal{E}}$, the curve $\pi^*(C)$ intersects every fiber once. This gives a section of~$\nu$. From \cite[Lemma~7.11]{MW19} it follows that the sections on $\mathcal{E}$ that come from exceptional curves on $S$ are exactly those that are disjoint from $\tilde{\mathcal{O}}$. 
\end{remark}

Let $\zeta_3\in \overline{k}$ be a primitive third root of unity. Note that, for a curve $C_R$ as in Definition~\ref{defncurve}, the morphism of $\mathbb{P}(2,3,1,1)$ given by multiplying the $w$-coordinate with $\zeta_3^2$ restricts to an automorphism of $\overline{C}_R$. 

\begin{definition}\label{defautomorphism}
Let $R=(x_R:y_R:z_R:1)$ be a point in $\mathcal{E}$, with $y_R,z_R\neq0$, and let $C_R$ be the corresponding curve as in Definition \ref{defncurve}. 
By $\sigma$ we denote the automorphism of $\overline{C}_R$ given by
\begin{equation}\label{definitionsigma}\sigma\colon(x:y:z:w)\longmapsto(x:y:z:\zeta_3^2w)=(\zeta_3^2x:y:\zeta_3 z:w)
\end{equation}
\end{definition}

Recall that $\pi\colon\mathcal{E}\longrightarrow S$ is the blow-up of $S$ in $\mathcal{O}$, and $\nu\colon\mathcal{E}\longrightarrow\mathbb{P}^1$ is the elliptic fibration on $\mathcal{E}$.

\begin{proposition}\label{singularpointsC}Let $R=(x_R:y_R:z_R:1)$ be a point in $\mathcal{E}$, with $x_R\in k$, $y_R,z_R\in k$ non-zero, and let $C_R$ be the curve in Definition \ref{defncurve}. The following hold.
\begin{itemize}
\item[](i) The curve $C_R$ is singular in $R$, $\sigma(R)$, and $\sigma^2(R)$. 
\item[](ii) If $\pi(R)$ is not contained in an exceptional curve on $\overline{S}=S\times_k\overline{k}$, then $C_R$ either contains a section that is defined over $k$, or it is geometrically integral and has geometric genus at most~1, in which case $R$, $\sigma(R)$, $\sigma^2(R)$ are all double points.
\end{itemize}
\end{proposition}
\begin{proof}
(i). It is an easy check that $R$ is contained in $C_R$. Let $m_R$ be the maximal ideal in the local ring of $R$ on $\mathcal{E}$. The point $R$ lies in $\mathbb{A}^3\subset \mathbb{P}(2,3,1,1)$ defined by $w\neq0$ as in Remark \ref{affinespace}. The ideal $m_R$ is generated by $X-x_R$, $Y-y_R$, and $T-z_R$. Let $F,G$ be as in (\ref{Caffien}).
We have $\mathcal{E}\cap\mathbb{A}^3=Z(F)$, and using the identity $c=y_R^2-x_R^3-az_R^6-bz_R^3$, we can write $F$ as \begin{multline*}
F = 2y_R(Y-y_R)- 3x_R^2(X-x_R)- (6az_R^5+3bz_R^2)(T-z_R)\\
+(Y-y_R)^2-(X-x_R)^3 - 3x_R(X-x_R)^2 - a(T-z_R)^6 - 6az_R(T-z_R)^5 \\
- 15az_R^2(T-z_R)^4 - (20az_R^3+b)(T-z_R)^3 - (15az_R^4+3bz_R)(T-z_R)^2.
\end{multline*}Set $$\alpha=2y_R(Y-y_R)- 3x_R^2(X-x_R)- (6az_R^5+3bz_R^2)(T-z_R),$$then it follows that $\alpha $ is contained in $m_R^2$, and the tangent line to $\mathcal{E}$ at $R$ is given by $\alpha=0$.

Similarly, we can rewrite $G$ as  
$$
G=-z_R^3\alpha
+3x_R^2z_R^2(X-x_R)(T-z_R) - (x_R^3-2az_R^6-bz_R^3)(T-z_R)^3 - (3x_R^3z_R-6az_R^7-3bz_R^4)(T-z_R)^2,
$$
so we conclude that $G$ is contained in $m_R^2$, hence $C_R$ is singular in $R$. Since $\sigma$ is an automorphism of $C_R$, this implies that $C_R$ is singular in $\sigma(R)$ and $\sigma^2(R)$, as well. \\
(ii). Assume that $\pi(R)$ is not contained in an exceptional curve on $\overline{S}$. We distinguish two cases. First assume that $\overline{C}_R$ is not irreducible or not reduced. Since $\overline{C}_R$ does not contain a fiber and intersects every fiber with multiplicity 3 (Lemma \ref{basicfactscr}), this implies that there is a curve that intersects every fiber with multiplicity one (hence is a section), say $H_1$, such that $\overline{C}_R$ either contains $H_1$ as irreducible component, or $\overline{C}_R$ is a multiple of $H_1$. Since $\overline{C}_R$ is disjoint from the zero section, it follows that $\pi(H_1)$ is an exceptional curve on $\overline{S}$ (Remark \ref{exccurvessections}). Therefore, by our assumption, $R$ is not contained in $H_1$, so $\overline{C}_R$ is not a multiple of $H_1$, and $H_1$ is an irreducible component of $\overline{C}_R$. Let $H_2$ be the other (not necessarily irreducible or reduced) component of $\overline{C}_R$, which contains $R$. If $H_2$ were not irreducible or not reduced, it would either be a double section or two sections intersecting in $R$. In both cases, $\pi(R)$ lies on an exceptional curve, contradicting our first assumption. We conclude that $H_2$ is irreducible and reduced. Since $C_R$ is defined over $k$, it is fixed by the action of the absolute Galois group of $k$ on Pic $S$. The exceptional curves of $\overline{S}$ are all defined over the separable closure $k^{\mbox{\tiny{sep}}}$ of $k$ by Remark \ref{exccurvessections}, so the Galois group $\mathrm{Gal}(k^{\mbox{\tiny{sep}}}/k)$ acts on them. Since $\overline{C}_R$ contains only one exceptional curve of $\overline{S}$, which is $H_1$, it follows that this component is invariant under the Galois action, hence it is defined over $k$. This finishes the first case. Now assume that $\overline{C}_R$ is irreducible and reduced. Since $C_R$ is contained in the linear system $|{-3K}_{\mathcal{E}}+3\tilde{\mathcal{O}}|$ by Lemma~\ref{basicfactscr}, from the adjunction formula it follows that its arithmetic genus is $\frac12\cdot(9-3)+1=4.$ Since the three distinct points $R$, $\sigma(R)$, $\sigma^2(R)$ are all singular on $C_R$ with the same multiplicity, we conclude that they all have multiplicity~2, and the geometric genus of $C_R$ is at most 1. 
\end{proof}

\begin{remark}
Proposition \ref{singularpointsC} (i) also follows from the description of $C_R$ in Remark \ref{kollar}. The curve $C_Q$ defined there is clearly singular in $Q$, as it is contained in the tangent plane to $T$ at $Q$. The inverse image of $C_Q$ under $f$ contains the 3 singular points $(x',y',w')$ for which $w'^3=w_R$, which are exactly the points $R$, $\sigma(R)$, and $\sigma^2(R)$. 
\end{remark}

\begin{remark}\label{genus1}In the last proof, we concluded that in the case where $C_R$ is geometrically integral, the geometric genus of $C_R$ is at most 1. If it were 0, then $C_R$ would contain exactly one more singular point besides $R$, $\sigma(R)$, $\sigma^2(R)$, say $Q$. Then $\sigma(Q)$ and $\sigma^2(Q)$ would be singular points of $C_R$ as well, so $Q$ would be a fixed point of~$\sigma$. 
Note that the points on the intersection of $C_R$ with the fiber above $(1:0)$ are fixed points of~$\sigma$. Assume that $\sigma$ has a fixed point $Q=(x_Q:y_Q:z_Q:1)$ in $C_R\setminus(C_R\cap\mathcal{E}_{(1:0)})$. From (\ref{definitionsigma}) it follows that there is a $\lambda\in k$ such that $\lambda^3y_Q=y_Q,$ $\lambda^2x_Q=x_Q,$ $\lambda z_Q=z_Q,$ and $\lambda\zeta_3^2=1.$ The last equation implies $\lambda=\zeta_3^{3n-2}$ for some $n>0$, and it follows that $x_Q=z_Q=0$. From the fact that $Q$ lies in $\mathcal{E}$ is follows that we have $
y_Q^2=c.$ We conclude that if $C_R$ is geometrically integral, then it has genus 0 if and only if it has a singular point which is a fixed point of $\sigma$ and which either lies on the fiber~$(1:0)$, or is the point $(0:\sqrt{c}:0:1)$. In our experiments with different surfaces and different points we have not found an example where this happens. \end{remark}

\begin{prooffamilymultisections}
Let $\mathcal{F}$ be as in the proposition. There are finitely many points on $\mathcal{F}$ with $y$-coordinate 0, and from Remark \ref{exccurvessections} it follows that there are at most 240 points $R$ on $\mathcal{F}$ with $\pi(R)$ contained in an exceptional curve on $\overline{S}$; let $V_1$ be the set of points in $\mathcal{F}$ for which either of these two conditions holds. Since $\mathcal{F}$ has positive rank over $k$ by assumption, the set of $k$-rational points on $\mathcal{F}_0=\mathcal{F}\setminus V_1$ is infinite. For each point $R\in\mathcal{F}_0(k)$ we can define the curve $C_R$ as in Definition \ref{defncurve}, and $C_R$ contains either a section defined over $k$, or it is geometrically integral and has geometric genus at most 1 by Proposition \ref{singularpointsC}. Therefore, 
if there is no $R\in\mathcal{F}(k)$ such that $C_R$ either contains a section over $k$ or is geometrically integral of genus 0, then for all $R\in\mathcal{F}_0(k)$, the curve $C_R$ is geometrically integral of genus 1.  \qed
\end{prooffamilymultisections}

\section{An elliptic fibration with non-torsion section}\label{familyofmultisectionsnontorsion}
Let $\mathcal{F}$ be a smooth fiber on $\mathcal{E}$ above a point $(z_0:1)\in\mathbb{P}^1$ with $z_0\neq0$, such that $\mathcal{F}$ has positive rank over $k$. Assume that there is an open subset of $\mathcal{F}$, which we denote by $\mathcal{F}_0$, such that for every point $R\in\mathcal{F}_0$, the curve $C_R$ is well-defined and geometrically integral of genus 1. In this section we show that this implies that there is an elliptic fibration over $\mathcal{F}$ that admits a non-torsion section. 

\begin{remark}\label{3section} Let $R$ be a point in $\mathcal{F}_0$. Since $C_R$ intersects every fiber of $\nu$ in three points counted with multiplicity (Lemma \ref{basicfactscr}), it is a 3-section. Moreover, since $R$ is a double point on $C_R$ by Proposition \ref{singularpointsC}, there is a unique third point of intersection of $C_R$ with $\mathcal{F}$, say $Q$ (see Figure  \ref{fig:multisection}). Hence $E_R=(\tilde{C}_R,Q)$ is an elliptic curve, where $\tilde{C}_R$ is the normalization of $C_R$. The curve $E_R$ contains a rational point, which we denote by $D_R$, which is the sum of the points corresponding to $\sigma(Q)$ and $\sigma^2(Q)$ on $C_R$. 
\end{remark}

\begin{figure}[!h]
\begin{center}
\includegraphics[width=8cm]{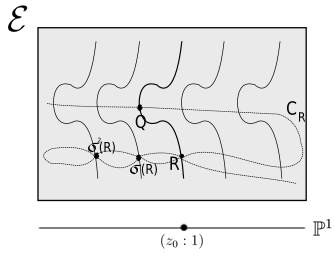}
  \caption{\label{fig:multisection}The multisection $C_R$ on $\mathcal{E}$, when $C_R$ is geometrically integral of genus 1.}
  \end{center}
\end{figure}

\begin{notation}\label{E en D}
For a point $R\in \mathcal{F}_0$, we denote by $E_R$ the corresponding elliptic curve and by $D_R$ the point on it, both as defined in Remark \ref{3section}. 
\end{notation}

Let $\eta$ be the generic point of $\mathcal{F}$, that is, $\eta$ is the point given by $(\tilde{x}:\tilde{y}:z_0:1)$ over the function field $k(\mathcal{F})=k(\tilde{x},\tilde{y})=\mbox{Frac}(k[x,y]/(y^2-x^3-az_0^6-bz_0^3-c))$ of $\mathcal{F}$. Let $C_{\eta}\subset\mathbb{P}_{k(\mathcal{F})}(2,3,1,1)$ be the corresponding curve given by (\ref{C}). From Proposition~\ref{singularpointsC} and Remark \ref{genus1} it follows that $C_{\eta}$ is geometrically integral of genus~1. Let $E_{\eta}$ be the corresponding elliptic curve with point $D_{\eta}$ as in Notation \ref{E en D}. 

\vspace{11pt}

In Lemma \ref{shortweierstrassandpoint} we give a Weierstrass model for the curve $E_{\eta}$, which we use in Proposition \ref{positiverank}. 
Recall that $a,b,c$, and $z_0$ are fixed elements in $k$, and $a,c,z_0$ are non-zero. We define the polynomial \begin{equation}\label{defnq}q=q_1q_2q_3q_4
\end{equation} in the polynomial ring $k[\tilde{x}]$ as follows. 
\begin{align*}
q_1&=\tilde{x};\\
q_2&=-\tilde{x}^6 +2z_0^3\left(4az_0^3+b\right)\tilde{x}^3+\left(4ac-b^2\right)z_0^6;\\
q_3&=\tilde{x}^6+8\left(az_0^6-c\right)\tilde{x}^3+8\left(2a^2z_0^{12}+3abz_0^9+\left(2ac+b^2\right)z_0^6+bcz_0^3\right);\\
q_4&=29\tilde{x}^{12}+(40c+24az_0^6)\tilde{x}^9+8(12a^2z_0^{12}+9abz_0^9+(18ac-5b^2)z_0^6-5bcz_0^3-2c^2)\tilde{x}^6 \\
&\;\;\;\;\;\;+32(4a^3z_0^{18}+9a^2bz_0^{15}+(5ab^2+
12a^2c)z_0^{12} + 14abcz_0^9 + (b^2c + 8ac^2)z_0^6 + bc^2z_0^3)\tilde{x}^3\\
&\;\;\;\;\;\;+16((4a^3c - a^2b^2)z_0^{18} + (8a^2bc - 2ab^3)z_0^{15}  + (8a^2c^2 + 2ab^2c - b^4)z_0^{12}  + (8abc^2- 2b^3c)z_0^9 \\
&\;\;\;\;\;\;\;\;\;\;\;\;\;\;\;\;\;\;\;\;\;\;\;\;\;\;\;\;\;\;\;\;\;\;\;\;\;\;\;\;\;\;\;\;\;\;\;\;\;\;\;\;\;\;\;\;\;\;\;\;\;\;\; \;\;\;\;\;\;\;\;\;\;\;\;\;\;\;\;\;\;\; \;\;\;\;\;\;\;\;\;\;\;\;\;\;\;\;\;\;\; \;\;\;\;\;\;\;\;\;\;\;\;\;\;\;\;\;\;\;   + (4ac^3 -b^2c^2)z_0^6).
\end{align*}

\begin{lemma}\label{shortweierstrassandpoint}
There exists a unique polynomial $\delta\in k[\tilde{x}]$ with leading term $-27cz_0^{48}\tilde{x}^{81}$ such that the following holds. There is an isomorphism $\omega$ between the elliptic curve $E_{\eta}$ and the curve with Weierstrass equation given by \begin{equation}\label{defnellcurve}\gamma^2=\xi^3+\delta,
\end{equation}
such that the denominators in the defining equations of $\omega$ and $\omega^{-1}$ are all of the form $2^p3^r(q_2q_4)^s$ for positive integers $p,r,s$. \\
The point on (\ref{defnellcurve}) corresponding to the point $D_{\eta}$ on $E_{\eta}$ is given by \begin{equation}\label{defnD_R}\omega(D_{\eta})=(\xi_{D},\gamma_{D}),
\end{equation}
where $$\xi_{D}=\tfrac{\alpha}{(q_1q_3)^2},\;\;\;\;\;\gamma_{D}=\tfrac{\beta}{(q_1q_3)^3},$$ are rational functions, and $\alpha$ and $\beta$ are polynomials in $k[\tilde{x}]$, with leading terms $\tfrac14z_0^{16}\tilde{x}^{42}$ and $\tfrac18z_0^{24}\tilde{x}^{63},$ respectively.
\end{lemma}
\begin{proof}The \texttt{magma} code that is used in this proof can be found in \cite{magmachapter2}. 
Let $Q$ be the third point of intersection of $C_{\eta}$ with the fiber of $\eta$ on the base change $\mathcal{E}\times_kk(\mathcal{F})$ over $\mathbb{P}^1\times_kk(\mathcal{F})$. Write $Q=(x_Q:y_Q:z_0:1)$, with $x_Q,y_Q\in k(\mathcal{F})$. 
Then $Q$ lies in $C_{\eta}\cap\left(\mathbb{A}^3\times_kk(\mathcal{F})\right)$, which is isomorphic to the curve $C_{\eta}^1$ in $\mathbb{A}^2\times_kk(\mathcal{F})$ defined by $H_{\eta}=0$, where $H_{\eta}$ is given in (\ref{H}) after substituting $R$ by~$\eta$. 
We find $x_Q$ by substituting $T=z_0$, $c=\tilde{y}^2-\tilde{x}^3-az_0^6-bz_0^3$ in (\ref{H}) and factorizing in $k(\mathcal{F})[X]$, which yields $$x_Q=\frac{9\tilde{x}^4-8\tilde{x}\tilde{y}^2}{4\tilde{y}^2}.$$ We conclude that the elliptic curve $E_{\eta}$ as defined in Remark \ref{3section} is isomorphic to the curve $\left(\tilde{C}_{\eta}^1,\left(\tfrac{9\tilde{x}^4-8\tilde{x}\tilde{y}^2}{4\tilde{y}^2},z_0\right)\right)$, where $\tilde{C}^1_{\eta}$ is the normalization of $C^1_{\eta}$. With \texttt{magma} we compute a Weierstrass model for $E_{\eta}$, which is given by \begin{equation}\label{weierstrassinproof}\gamma'^2=\xi'^3+\frac{(3\cdot2^5)^6\delta}{(q_2q_4)^6},
\end{equation} where $\delta$ is a polynomial in $k[\tilde{x}]$ with leading term $-27cz_0^{48}\tilde{x}^{81}$. We verify with \texttt{magma} that the denominators in the defining equations of the isomorphism $\omega_1$ between $E_{\eta}$ and the curve (\ref{weierstrassinproof}), as well as those of $\omega_1^{-1}$, are all of the form $2^{p'}(q_2q_4)^{s'}$ for positive integers $p',s'$. The change of coordinates $$\xi'=\tfrac{(3\cdot2^5)^2}{(q_2q_4)^2}\xi,\;\;\;\;\;\;\;\; \gamma'=\tfrac{(3\cdot2^5)^3}{(q_2q_4)^3}\gamma,$$ induces an isomorphism $\omega_2$ between the curve (\ref{weierstrassinproof}) and the curve defined by \begin{equation}\label{weierstrassinproofdelta}\gamma^2=\xi^3+\delta.
\end{equation} 
We conclude that $\omega=\omega_2\circ\omega_1$ is an isomorphism between $E_{\eta}$ and the curve (\ref{weierstrassinproofdelta}), and the denominators in the defining equations of $\omega$ and $\omega^{-1}$ are all of the form $2^p3^r(q_2q_4)^s$ for positive integers $p,r,s$. \\
If $\delta'$ was another polynomial in $k[\tilde{x}]$ such that $E_{\eta}$ were isomorphic to the curve given by $\gamma^2=\xi^3+\delta'$, then we would have $\delta'=\upsilon^6\delta$ for some $\upsilon\in k(\mathcal{F})$, hence $\delta'$ would not have leading term $-27cz_0^{48}\tilde{x}^{81}$. We conclude that $\delta$ is the unique polynomial with leading term $-27cz_0^{48}\tilde{x}^{81}$ such that $E_{\eta}$ is isomorphic to the curve with Weierstrass model (\ref{weierstrassinproofdelta}).
With \texttt{magma} we compute the sum $D$ on the curve (\ref{weierstrassinproofdelta}) of the points corresponding to $\left(\zeta_3^2\tfrac{9\tilde{x}^4-8\tilde{x}\tilde{y}^2}{4\tilde{y}^2},\zeta_3z_0\right)$ and $\left(\zeta_3\tfrac{9\tilde{x}^4-8\tilde{x}\tilde{y}^2}{4\tilde{y}^2},\zeta_3^2z_0\right)$ on $C_{\eta}$. We find $D=(\xi_{D},\gamma_{D})$ with $\xi_{D}=\tfrac{\alpha}{(q_1q_3)^2}$, $\gamma_{D}=\tfrac{\beta}{(q_2q_3)^3}$, where $\alpha,\beta$ are elements in $k[\tilde{x}]$ with leading terms given by $\tfrac14z_0^{16}\tilde{x}^{42}$ and $\tfrac18z_0^{24}\tilde{x}^{63}$, respectively. \end{proof}

\begin{remark}\label{specialization}
The curve in (\ref{defnellcurve}) over the function field $k(\mathcal{F})$ of $\mathcal{F}$ gives rise to a unique relatively minimal elliptic surface $\rho\colon \mathcal{C}\longrightarrow\mathcal{F}$ over $\mathcal{F}$  \cite[Theorem 5.19]{MW19}, such that the generic fiber of $\rho$ is isomorphic to $E_{\eta}$. Recall the polynomial $q$ in (\ref{defnq}). From Lemma \ref{shortweierstrassandpoint} it follows that for every $R=(x_R:y_R:z_0:1)\in\mathcal{F}_0$ with $q(x_R)\neq0$, the fiber of $\rho$ above $R$ is isomorphic to the curve $E_R$ as in Notation \ref{E en D}. Moreover, the point $D_{\eta}$ on $E_{\eta}$ gives rise to a section $D$ on $\mathcal{C}$. See Figure \ref{fiberfibration}. In Proposition \ref{positiverank}, we show that $D$ is a non-torsion section.
\end{remark}

\begin{figure}[!h]
\captionsetup{width=0.8\textwidth}
\begin{center}
 \includegraphics[width=13cm]{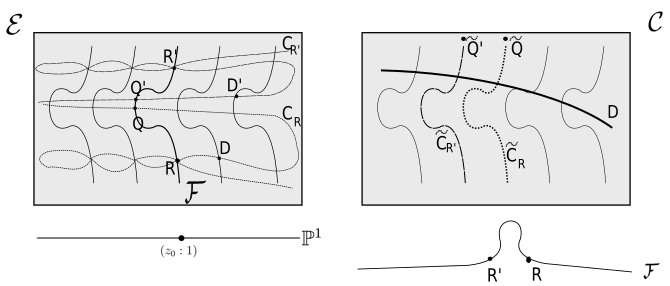}
  \caption{\label{fiberfibration}\textsl{Left:} two points $R,R'\in\mathcal{F}_0(k)$, with corresponding curves $C_R,C_{R'}$. \textsl{Right:} the fibration on $\mathcal{F}$ with two fibers that are the normalizations of the multisections $C_R$, $C_{R'}$, and the section $D$. }
  \end{center}
\end{figure}

\begin{proposition}\label{positiverank}
If $k$ has characteristic 0, then the point $D_{\eta}$ is non-torsion on $E_{\eta}$. 
\end{proposition}

\begin{proof}
It suffices to show that the section $D$ intersects a fiber of $\rho\colon \mathcal{C}\longrightarrow\mathcal{F}$ in a non-torsion point. We use the Weierstrass equation for the elliptic surface $\mathcal{C}$ given in Lemma \ref{shortweierstrassandpoint}, and look at the fiber above the point at infinity on $\mathcal{F}$ by setting $\psi=\tfrac{1}{\tilde{x}}$, multiplying by factors of $\psi$ to obtain polynomials in $k[\psi]$, and evaluating at $\psi=0$. \\
After setting $\psi=\tfrac{1}{\tilde{x}}$ and applying the change of coordinates $\xi'=\psi^{28}\xi,\;\gamma'=\psi^{42}\gamma$ in (\ref{defnellcurve}) we obtain $\gamma'^2=\xi'^3+\delta',$ where $\delta'=\psi^{84}\delta$. 
Since the leading term of $\delta$ in $k[\tilde{x}]$ has degree 81, the rational function $\delta'$ is in fact a polynomial in $k[\psi]$, and it is divisible by $\psi^3$. Therefore, evaluating $\delta'$ at $\psi=0$ gives $0$, and the fiber above infinity on $\mathcal{F}$ is given by $\mathcal{F}_{\infty}:\gamma'^2=\xi'^3$. \\
The section $D$ intersects $\mathcal{F}_{\infty}$ in a point $(\xi_{\infty},\gamma_{\infty})$. The leading coefficient of $(q_1q_3)^2$ equals $\tilde{x}^{14}$, so we have $$\xi_{D}=\frac{\alpha(\frac{1}{\psi})}{(q_1(\frac{1}{\psi})q_3(\frac{1}{\psi}))^2}=\frac{\psi^{14}\alpha(\frac{1}{\psi})}{\psi^{14}(q_1(\frac{1}{\psi})q_3(\frac{1}{\psi}))^2},$$ where the denominator of the latter is a polynomial $q'$ in $k[\psi]$ with constant coefficient 1, and the numerator, as polynomial in $\tilde{x}$, has leading term $\frac14z_0^{16}\tilde{x}^{28}$. We conclude that we have $$\xi_{\infty}=(\psi^{28}\xi_D)|_{\psi=0}=\frac{\frac14z_0^{16}}{q'(0)}=\frac14z_0^{16}.$$ Similarly, we have $\gamma_{\infty}=\frac18z_0^{24}$. So $D$ intersects $\mathcal{F}_{\infty}$ in the non-zero point $(\frac14z_0^{16},\frac18z_0^{24})$. Since the group $\mathcal{F}_{\infty}^{\rm{ns}}(k)$ of non-singular $k$-rational points on $\mathcal{F}_{\infty}$ is isomorphic to the additive group of $k$, and we have char $k=0$, we conclude that $(\frac14z_0^{16},\frac18z_0^{24})$ is non-torsion on $\mathcal{F}_{\infty}$.\end{proof}

\section{Proof of Theorem \ref{thm1}}\label{proofmainthm}
In this section we prove Theorem \ref{thm1}. Let $a, b, c$, $k$, $S$, and $\mathcal{E}$ be as in the theorem; in particular, $k$ is now a field of characteristic 0. Recall Notation \ref{setup}. 

\vspace{11pt}

We use the following theorem of Colliot-Th\'{e}l\`{e}ne, which gives us a stronger version of Merel's theorem for bounding the torsion in a family of elliptic curves. The proof can be found in the Appendix at the end of this paper. The same result is also mentioned in Footnote 1 of the paper \cite{CaTa} by Cadoret and Tamagawa.

\begin{theorem}\label{CT}[Th\'{e}or\`{e}me in the Appendix]
Let $l$ be a field that is finitely generated over $\mathbb{Q}$. There exists an integer $N=N(l)$ such that if $C/l$ is a geometrically integral $l$-variety, and $E/C$ a smooth family of elliptic curves, then for all points $P\in C(l)$, the order of a torsion point on the fiber $E_P(l)$ is at most $N$.
\end{theorem}

\begin{proofofmainthm}
By assumption, there is a point $P\in S(k)$ such that the corresponding point $P_{\mathcal{E}}$ on $\mathcal{E}$ lies on a smooth fiber $\mathcal{F}_P$ above $(z_P:1)\in\mathbb{P}^1$ for some $z_P\in k$ non-zero, and $P_{\mathcal{E}}$ is non-torsion on $\mathcal{F}_P$. From Proposition \ref{multisection} it follows that at least one of the following holds. 
\begin{itemize}
\item[](i) There is a point $R\in\mathcal{F}_P(k)$ such that the curve $C_R$ as in (\ref{C}) contains a section defined over $k$;
\item[](ii) there is a point $R\in\mathcal{F}_P(k)$ such that the curve $C_R$ as in (\ref{C}) is geometrically integral of genus 0;
\item[](iii) there is an open subset $\mathcal{F}_{P,0}$ of $\mathcal{F}_P$ such that for every point $R\in\mathcal{F}_{P,0}(k)$, the curve $C_R$ is geometrically integral of genus 1.
\end{itemize}
Note that in case (i) we are done by Remark \ref{sectionimpliesdensity}. In case (ii), the normalization $n\colon \tilde{C}_R\longrightarrow C_R$ gives a smooth curve of genus 0. Since $R$ is not a triple point on $C_R$ (Proposition \ref{singularpointsC}), the latter contains a rational point given by the unique other point in the intersection of $C_R$ with $\mathcal{F}_P$, hence $C_R$ contains infinitely many $k$-rational points. Consider the base change $\nu_R\colon \mathcal{E} \times_{\mathbb{P}^1} \tilde{C}_R \longrightarrow \tilde{C}_R$ of $\nu$, which is an elliptic fibration with section $\tilde{C}_R\longrightarrow \mathcal{E} \times_{\mathbb{P}^1} \tilde{C}_R,\; p\longmapsto (n(p),p)$. The latter has infinite order \cite[Theorem 6.4]{SL14}, and is defined over $k$, so the set $(\mathcal{E} \times_{\mathbb{P}^1} \tilde{C}_R)(k)$ is dense in $\mathcal{E} \times_{\mathbb{P}^1} \tilde{C}_R$ by Remark \ref{sectionimpliesdensity}. Since $\mathcal{E} \times_{\mathbb{P}^1} \tilde{C}_R$ maps dominantly to $\mathcal{E}$, and hence to $S$, the density of $S(k)$ in $S$ follows. If we are in case (iii), it follows from Remark \ref{specialization} and Proposition \ref{positiverank} that there is an elliptic fibration $\rho\colon\mathcal{C}_P\longrightarrow \mathcal{F}_P$ such that almost all fibers are the normalizations of a 3-section of $\mathcal{E}$, and such that $\rho$ admits a section defined over $k$ of infinite order. From Remark \ref{sectionimpliesdensity} it follows that $\mathcal{C}_P(k)$ is Zariski dense in $\mathcal{C}_P$. Since almost all fibers of $\rho$ are normalizations of 3-sections of $\mathcal{E}$, the surface $\mathcal{C}_P$ maps dominantly to $\mathcal{E}$, and hence to $S$. It follows that $S(k)$ is dense in $S$ as well. This proves the first statement of the theorem.\\
Now assume that $k$ is finitely generated over $\mathbb{Q}$, and that $S(k)$ is dense in $S$. Since every smooth fiber of $\mathcal{E}$ is an elliptic curve over $k$, there is an upper bound $N=N(k)$ such that on all the fibers, all the torsion points have order at most $N$ (Theorem \ref{CT}). Let $m\leq N$ be an integer, and let $T_m$ be the zero locus of the $m$-th division polynomial $\psi_m\in k[x,y,t]$ of the generic fiber $E$ of $\mathcal{E}$, which is an elliptic curve over the function field $k(t)$ of $\mathbb{P}^1$. We have $\psi_m\in k[x,t]$, and for any $\tau\in k$, the polynomial $\psi_m(x,\tau)\in k[x]$ has degree $m^2$ \cite[Exercise III.3.7]{Sil09}. So $T_m$ is an $m^2$-section of $\mathcal{E}$. If $S$ did not contain a point $P$ as in the theorem, then $S(k)$ would be contained in the union of the torsion locus $\cup_{m\leq N}T_m$ with the two fibers $(1:0)$ and $(0:1)$ and the singular fibers, which is a strict closed subset of $S$, contradicting the assumption that $S(k)$ is dense in $S$. This finishes the proof. \qed
\end{proofofmainthm}

\section{Examples}\label{examples}
We conclude this paper by giving examples where we prove the density of rational points on specific surfaces. The rank of the Mordell--Weil group over $\mathbb{Q}$ of the surfaces in Examples \ref{exsurface1} and \ref{exsurface2} is 0 by \cite[Corollary 2.4]{DN}, so in these cases the density of the $\mathbb{Q}$-rational points can not be proven by the existence of a section over $\mathbb{Q}$ (see also Remark \ref{sectionimpliesdensity}). The surface in Example \ref{ex3} has Mordell-Weil rank 2 over $\mathbb{Q}$ \cite[Corollary 2.4]{DN}, so the density of the set of rational points is implied by Remark~\ref{sectionimpliesdensity}. For this surface, we show how our method can construct a rational section as a component of the curve $C_R$ for a certain point $R$ (this is one of the cases in Proposition \ref{multisection}).

\begin{example}\label{exsurface1}Let $k$ be field of characteristic 0 and let $S$ be the surface given by  
$$y^2=x^3+6(27z^6+w^6).$$Note that $S$ does not satisfy the conditions of \cite[Theorem 1.1]{VA11} since $3\cdot27$ is a square and gcd$(6\cdot27,6)\neq1$, hence the density of $\mathbb{Q}$-rational points could not be proven by V\'{a}rilly-Alvarado \cite[Example 7.2]{VA11}. However, the fiber $\mathcal{E}_{(1:1)}$ of the anticanonical elliptic surface $\mathcal{E}$ above $(1:1)$ is smooth, and with \texttt{magma} we find that this fiber has rank 2. So $S$ contains a point that lies on a smooth fiber of $\mathcal{E}$ and has infinite order, hence $S(k)$ is dense in $S$ by Theorem \ref{thm1}.\\
We illustrate this by constructing a 3-section as in (\ref{C}). With \texttt{magma} we find two generators for $\mathcal{E}_{(1:1)}(\mathbb{Q})$, given by $P_1=(1:13:1:1)$ and $P_2=(22:104:1:1)$. The curve $C_{P_1}$ is cut out from $S$ by $3xz-26y+323z^3+12w^3$, and it has geometric genus 1. We find $C_{P_1}\cap\mathcal{E}_{(1:1)}=\{P_1,Q_1\}$ with $Q_1=\left(-\frac{1343}{676} : \frac{222431}{17576} : 1 : 1\right)$. The elliptic curve $E=(\tilde{C}_{P_1},Q_1)$ is given by Weierstrass equation $$\gamma^2 = \xi^3 -2\cdot3^{4}\cdot5^2\cdot28368481,$$
and the point $D=\sigma(Q_1)+\sigma^2(Q_1)$ has infinite order on $E$; its $\xi$-coordinate is given by $$\xi_D=\frac{11\cdot33487\cdot
580020724757}{(2\cdot12\cdot167\cdot523)^2},$$ so $D$ has infinite order on $E$ by a result of Lutz and Nagel ([Corollary VIII.7.2]\cite{Sil09}). We conclude that the 3-section $C_{P_1}$ has infinitely many $k$-rational points.
Equivalently, we could have used the point $P_2$ to create a 3-section with infinitely many $k$-rational points: the curve $C_{P_2}$ is cut out from $S$ by $1452xz - 208y - 10324z^3 + 12w^3$; it has geometric genus 1, the third point of intersection of $C_{P_2}$ with the fiber $\mathcal{E}_{(1:1)}$ is given by $Q_2=\left(\frac{12793}{2704}:-\frac{2327053}{140608}:1:1\right),$ and the point $\sigma(Q_2)+\sigma^2(Q_2)$ again has infinite order on the elliptic curve $(\tilde{C}_{P_2},Q_2)$. We conclude that also $C_{P_2}$ has infinitely many $k$-rational points. 
\end{example}

\begin{example}\label{exsurface2}Let $k$ be a field of characteristic 0 and consider the surface $S$ given by $$y^2=x^3+243z^6+16w^6.$$ Note that this surface does not satisfy the conditions of \cite[Theorem 1.1]{VA11}, so the method there failed in this case \cite[Remark 7.4]{VA11}. Salgado and van Luijk made the observation that this surface contains the point $P=(0:4:0:1)$, which is 3-torsion on its fiber on the corresponding elliptic surface $\mathcal{E}$ (more generally, for $\beta\in k^*$, the elliptic curve of the form $y^2=x^3+\beta^2$ has the 3-torsion point $(0,\beta)$). However, this point is contained in 9 exceptional curves, so their method does not work with $P$. They did not find another point for which the computations were doable to show density of $S(k)$ \cite[Examples 7.3 and 4.4 (iii)]{SL14}. Finally, Elkies showed that the set $S(\mathbb{Q})$ is Zariski dense in $S$, by constructing a multisection with infinitely many rational points in the linear system $|{-3K}_S|$ that contains $P$ as a point of multiplicity 3 (this idea was generalized to any surface with a torsion point in the master thesis \cite{BvL}, though under the assumption that at least one of the infinitely many multisections constructed there has infinitely many rational points). \\
We prove the density of $S(k)$ in $S$ using Theorem \ref{thm1}: with \texttt{magma} we find that the fiber $\mathcal{E}_{(1:5)}$ above $(1:5)$ is smooth and has rank 2, so $S$ contains a point that lies on a smooth fiber of $\mathcal{E}$ and has infinite order (for example $P=(-63:-14:1:5)$), hence $S(k)$ is dense in $S$. 
\end{example}

\begin{example}\label{ex3}
Let $k$ be a field of characteristic 0, and let $S$ be the del Pezzo surface of degree 1 over $k$ given in $\P(2,3,1,1)$ by the equation

$$y^2=x^3+27z^6+16w^6.$$ For $k=\mathbb{Q}$, the rank of the Mordell--Weil group of the corresponding elliptic surface $\mathcal{E}$ is 2 \cite[Corollary 2.4]{DN}. In this example we illustrate that different cases in Proposition \ref{multisection} can happen on the same surface; we give a point $P$ on $S$ such that $C_P$ contains a section defined over $k$, and we give a point $Q$ on $S$ such that $C_Q$ is geometrically integral of genus~1. We have not found a point for which the corresponding curve is geometrically integral of genus 0 (see also Remark \ref{genus1}).\\
The point $P=(-3:-4:1:1)$ on $S$ corresponds to a non-torsion point on a smooth fiber of~$\mathcal{E}$. The curve $C_P$ is cut out from $S$ by the hypersurface $27xz + 8y + 81z^3 + 32w^3=0$. It is the union of a 2-section of genus 1 and a section of genus 0, both containing $P$. For $k=\mathbb{Q}$ we compute with \texttt{magma} that the 2-section is an elliptic curve of rank 3. The section is given by the curve $x+3z^2=y+4w^3=0$ in $\mathbb{P}(2,3,1,1)$, and it corresponds to the point $(-3t^2,-4)$ on the generic fiber of $\mathcal{E}$ over the function field $k(t)$ of $\mathbb{P}^1$, where we set $t=z/w$. This is the first case in Proposition \ref{multisection}.\\
By starting with the point $Q=(36:-220:2:1)$ on $S$, which also corresponds to a point on a smooth fiber of $\mathcal{E}$, we obtain the curve $C_Q$ cut out on $S$ by the surface $243xz + 55y - 675z^3 + 4w^3$. The curve $C_Q$ is now geometrically integral of genus 1. for $k=\mathbb{Q}$ we compute with \texttt{magma}, under the condition \texttt{SetClassGroupBounds("GRH")}, that its normalization is an elliptic curve of rank 4. 
\end{example}

\section*{Appendix. Un corollaire d'un th\'{e}or\`{e}me de Merel (par Jean-Louis Colliot-Th\'{e}l\`{e}ne)}

Un th\'eor\`eme bien connu de L. Merel \cite{Merel} borne la torsion des courbes elliptiques sur un corps de nombres $k$, et ce de fa\c con uniforme en fonction uniquement du degr\'e du corps de $k$ sur $\mathbb{Q}$. Je remarque qu'on en d\'eduit facilement une extension au cas des corps de type fini sur $\mathbb{Q}$.

On utilise le lemme bien connu suivant.

\begin{lemme} Soient $k$ un corps de caract\'eristique z\'ero, $Y$ une k-vari\'et\'e int\`egre et $f : X \longrightarrow Y$ une famille lisse de vari\'et\'es ab\'eliennes. Si la fibre g\'en\'erique de f poss\`ede un point exactement de $n$-torsion, alors pour tout point (sch\'ematique) $P$ du sch\'ema $Y$, la fibre $X_P/\kappa(P)$ poss\`ede un point exactement de $n$-torsion.\end{lemme}

\textit{D\'emonstration.} Pour tout entier $n$, le sch\'ema des points de $n$-torsion est fini \'etale sur $Y$. En particulier le sous-sch\'ema form\'e des sections d'ordre exactement $n$ est une union disjointe d'images de sections de $f$. 
CQFD

\vspace{11pt}

Voici l'extension du th\'eor\`eme de Merel.

\begin{theoreme} Soit k un corps de type fini sur $\Q$, soit $C$ une k-vari\'et\'e int\`egre, et soit $E/C$ une famille lisse de courbes elliptiques. Alors il existe un entier $N$ (d\'ependant de $k$) tel que, pour tout point $P\in C(k)$, l'ordre d'un point $k$-rationnel de torsion sur la fibre $E_P$ est au plus $N$.\end{theoreme}

\textit{D\'emonstration.} Le corps $k$ s'\'ecrit comme le corps des fractions d'une $\Q$-vari\'et\'e int\`egre $U=\rm{Spec}(A)$, qu'on peut choisir finie \'etale, d'un certain degr\'e $d$, sur un ouvert d'un espace affine $\mathbb{A}^r_{\mathbb{Q}}$. Quitte \`a restreindre $U$ on peut \'etendre la situation $E/C/k/\mathbb{Q}$ \`a $F/D/U/\mathbb{Q}$ avec $F/D$ famille de courbes elliptiques sur $D$. Un point $k$-rationnel $P$ de $C$ s'\'etend en une section $\tau_P : V \longrightarrow D$ de la projection $D \longrightarrow U$ sur un ouvert $V\subset U$ (ouvert d\'ependant de $P$) non vide. L'image r\'eciproque de $F \longrightarrow D$ au-dessus de $V$ via $\tau_P$ est une famille de courbes elliptiques dont la fibre g\'en\'erique est $E_P$. L'ensemble des points ferm\'es de $V$ de degr\'e au plus $d$ est Zariski dense dans $V$ (consid\'erer les images r\'eciproques des points de $\mathbb{A}^r(\mathbb{Q}))$, en particulier est non vide. Le th\'eor\`eme de Merel \cite{Merel} assure que l'ordre des points de torsion des courbes elliptiques sur un corps de nombres de degr\'e au plus $d$ est born\'e par un entier $N(d)$. Le lemme permet alors de conclure. CQFD

\vspace{11pt}

Si l'on note $\phi(d)$ la borne sur l'ordre d'un point de torsion donn\'ee par le th\'eor\`eme de Merel sur les corps de nombres de degr\'e au plus $d$ et si, pour $k$ de type fini sur $\mathbb{Q}$, on note $d_{min}(k)$ le degr\'e minimal de la pr\'esentation de $k$ comme extension finie $k/E$ d'une extension transcendante pure $E$ de $\mathbb{Q}$, alors on peut borner $N$ dans le th\'eor\`eme par $\phi(d_{min}(k))$.

\bibliographystyle{amsalpha}
\bibliography{DensityDP1}

\end{document}